\documentclass[12pt, reqno]{amsart}

\usepackage{amsmath,amssymb,amscd,amsxtra}
\usepackage{latexsym}
\usepackage[dvips]{graphics,epsfig}
\usepackage{enumerate}

\usepackage[margin=1.2in,marginparwidth=1.5cm, marginparsep=0.5cm]{geometry}

\usepackage
[implicit=true]{hyperref}

\newtheorem{theorem}{Theorem}[section]
\newtheorem{corollary}[theorem]{Corollary}
\newtheorem{lemma}[theorem]{Lemma}
\newtheorem{proposition}[theorem]{Proposition}
\newtheorem{definition}[theorem]{Definition}
\newtheorem{remark}[theorem]{Remark}
\newtheorem{example}[theorem]{Example}

\newcommand{\al}{\alpha}
\newcommand{\be}{\beta}

\newcommand{\cT}{\mathcal T}
\newcommand{\Si}{\Theta}
\newcommand{\om}{\omega}

\newcommand{\bn}{{\mathbb N}}
\newcommand{\br}{{\mathbb R}}

\newcommand{\rd}{{\mathbb R^d}}
\newcommand{\bz}{{\mathbb Z}}

\newcommand{\zd}{\mathbb Z^d}

\newcommand{\elp}{\ell^p}
\newcommand{\elq}{\ell^q}
\newcommand{\elr}{\ell^r}

\newcommand{\N}{\mathbb{N}}

\newcommand{\beq}{\begin{equation}}
\newcommand{\eeq}{\end{equation}}

\newcommand{\eps}{\varepsilon}

\newtheorem*{ackno}{Acknowledgments}

\usepackage{color}

\newcommand{\les}{\lesssim}

\newcommand{\noi}{\noindent}

\newcommand{\ind}{\mathbf 1}

\newcommand{\Z}{\mathbb{Z}}
\newcommand{\R}{\mathbb{R}}
\newcommand{\C}{\mathbb{C}}
\newcommand{\T}{\mathbb{T}}

\newcommand{\s}{\sigma}

\newcommand{\F}{\mathcal{F}}

\renewcommand{\l}{\ell}

\newcommand{\Dl}{\Delta}

\newcommand{\dd}{\partial}

\newcommand{\jb}[1]
{\langle #1 \rangle}

\newcommand{\ft}{\widehat}

\numberwithin{equation}{section}
\numberwithin{theorem}{section}

\makeatletter
\@namedef{subjclassname@2020}{%
  \textup{2020} Mathematics Subject Classification}
\makeatother

\begin{document}

\baselineskip = 14 pt

\title[Discrete bilinear operators]
{Discrete bilinear operators and commutators}

\author[\'A. B\'enyi and T. Oh]{\'Arp\'ad B\'enyi and Tadahiro Oh}


\subjclass[2020]{Primary 35S05, 47G30; Secondary 42B15, 42B20}

\keywords{bilinear discrete operator; bilinear pseudodifferential operator; bilinear
H\"ormander class; bilinear Calder\'on-Zygmund kernel}

\address{\'Arp\'ad B\'enyi, Department of Mathematics,
516 High St, Western Washington University, Bellingham, WA 98225,
USA.} \email{benyia@wwu.edu}

\address{Tadahiro Oh, School of Mathematics, The University of Edinburgh, and The Maxwell
Institute for the Mathematical Sciences, James Clerk Maxwell Building, The King’s Buildings,
Peter Guthrie Tait Road, Edinburgh, EH9 3FD, UK.}
\email{hiro.oh@ed.ac.uk}


\begin{abstract}
We discuss  boundedness properties of certain classes of discrete bilinear operators that are similar to those of the continuous bilinear pseudodifferential operators with symbols in the H\"ormander classes $BS^{\om}_{1, 0}$.  In particular, we investigate their relation to discrete analogues of the bilinear Calder\'on-Zygmund singular integral operators and show   compactness of their commutators.
\end{abstract}

\maketitle

\section{Introduction}

The study of bilinear operators within Fourier analysis goes back to the seminal work of Coifman and Meyer in the 1970's on the Calder\'on commutators \cite{cm0, cm1, cm2}. A nowadays classical result, known as the Coifman-Meyer multiplier theorem, concerns the boundedness of bilinear pseudodifferential operators with symbols in the class $BS_{1, 0}^0(\R^d)$ from $L^p(\R^d)\times L^q(\R^d)$ to $L^r(\R^d)$ for $1 < p, q\le \infty$,
$\frac 12 < r < \infty$,  and $\frac{1}{p}+\frac{1}{q}=\frac{1}{r}$. This further led to the general study of the bilinear H\"ormander classes $BS^{\om}_{\rho, \delta}$ including, when appropriate, their symbolic calculus, boundedness properties, and applications to PDEs; see, for example, \cite{ bbmnt, bmnt,N, T} and the references therein for a succinct introduction to the subject. More precisely, the bilinear pseudodifferential operators considered in these works are a priori defined on appropriate test functions and
are given by
\beq\label{continuous-bilinearop}
T_\sigma (f, g)(x)=\int_\rd \int_{\mathbb R^d} \sigma (x, \xi, \eta)\widehat
f(\xi)\widehat g(\eta)e^{ix\cdot (\xi+\eta)}\, d\xi d\eta,
\eeq
where the symbol $\sigma$ satisfies estimates of the form:
\begin{equation}\label{hormander}
|\partial_x^\alpha\partial_\xi^\beta\partial_\eta^\gamma \sigma (x,
\xi, \eta)|\leq C_{\alpha, \beta, \gamma} (1+|\xi|+|\eta|)^{\om+\delta
|\alpha|-\rho (|\beta|+|\gamma|)},
\end{equation}
for any $x,\,\xi,\,\eta\in\mathbb R^d$, any multi-indices $\alpha,
\beta, \gamma$,  and some positive constants
$C_{\alpha, \beta, \gamma} > 0$.
The  class of  symbols satisfying
\eqref{hormander} is denoted by $BS_{\rho, \delta}^{\om} (\mathbb R^d),$ or simply
$BS_{\rho, \delta}^{\om}$ when it is clear from the context to which space the variables
$x, \xi, \eta$ belong. For example, any bilinear partial differential operator with variable coefficients that have bounded derivatives
\begin{align}
L(f, g)=\sum_{|\al|+|\be|\leq\om}a_{\al, \be}(x)\frac{\partial^\al f}{\partial x^\al}\frac{\partial^\be g}{\partial x^\be}
\label{D1}
\end{align}
can be realized as an operator of order $\om$
with $L=T_\sigma$ as in \eqref{continuous-bilinearop},  where $\sigma\in BS^{\om}_{1, 0}$.

Despite  H\"ormander's initial undertakings \cite{H} ``the use of Fourier transformations has been emphasized; as a result no singular integral operators are apparent...", the later works made clear that,  for symbols of order zero such as $BS_{1, 0}^0$, the associated operators are bilinear Calder\'on-Zygmund singular integral operators, and therefore appropriate tools from the theory of singular integrals
provide  an alternative argument for  boundedness results without the use of Littlewood-Paley theory as in \cite{cm1}; see also~\cite{bo}. More precisely, if we formally invert the Fourier transforms in~\eqref{continuous-bilinearop}, $T_\sigma$ has an integral representation on the physical side  as
\[
T_\sigma (f, g)(x)=\langle K_\sigma(x, y, z), (f\otimes g)(y, z)\rangle,
\]
where $(f\otimes g)(y, z)=f(y)g(z)$, $\langle\cdot,\cdot\rangle$ denotes the usual distribution-test function pairing (in $y$ and $z$), and the kernel $K_\sigma (x, y, z)$ is an appropriate distribution that is singular along some variety.
It is known \cite{GT2, BT} that if the symbol $\sigma\in BS_{1, 0}^0$, then $K_\sigma$ is a bilinear Calder\'on-Zygmund kernel, that is, a function on $\br^{3d}$ such that, away from the diagonal $\mathbb D=\{(x, y, z)\in\br^{3d}: x=y=z\}$, we have
\begin{align*}
|K_\sigma(x, y, z)|&\lesssim (|x-y|+|y-z|+|z-x|)^{-2d},\\
|\nabla K_\sigma (x, y, z)|&\lesssim (|x-y|+|y-z|+|z-x|)^{-2d-1},
\end{align*}
and
\beq
\label{kernel}
T_\sigma(f, g)(x)=\int_{\rd}\int_{\rd}K_\sigma(x, y, z)f(y)g(z)\,dydz
\eeq
for $x\not\in\text{supp}(f)\cap\text{supp}(g)$. Moreover,
\cite[Theorem 5.1]{bmnt} shows that, given  multi-indices $\alpha, \beta, \gamma
\in(\bn\cup\{0\})^d$, there exists $N_0\in\bn$ such that for all $N\geq N_0$,
\beq
\label{diag-esti}
\sup_{(x, y, z)\in\br^{3d}\setminus\mathbb D}|\partial_x^\alpha\partial_y^\beta\partial_z^\gamma K_\sigma(x, y, z)|
(|x-y|+|y-z|+|z-x|)^N<\infty.
\eeq
The goal of this note is to put forth the conceivability of an appropriate theory of \emph{discrete} bilinear operators of the same flavor as the continuous ones discussed above. As such, it is inspired by the works \cite{C} and \cite{fg}. The essential insight of \cite{C} was to characterize the space of infinite matrices that model (linear) pseudodifferential operators and to define an appropriate notion of order that is reminiscent, if not the same, of that appearing in the definition of the linear H\"ormander class $S^{\om}_{1, 0}$. It is within this framework that several applications to numerical approximations in PDEs were then taken up in \cite{fg}. Our own modest intention instead stems purely from exploring some appropriate boundedness results for the discrete analogues of classes of operators defined via infinite Calder\'on-Zygmund-like tensors, as well as showing the compactness of such operators and multiplication by bounded sequences with compact support.

\section{Definitions}\label{Def}

A discrete weight function is a non-negative function $w: \mathbb Z^d\to [0, \infty)$.
Given $1\leq p\leq \infty$, the \emph{weighted space of $p$-summable sequences} $\ell^p_w(\mathbb Z^d)$ (or simply $\ell^p_w$ when the indexing set of the sequences is understood contextually) consists of all functions $f: \mathbb Z^d\to\mathbb C$, that is sequences $f=\{f_k\}_{k\in\mathbb Z^d}$, such that the norm
\begin{equation*}
\|f\|_{\ell^p_w}=\bigg(\sum_{k\in\mathbb Z^d}w(k)^p|f_k|^p\bigg)^{\frac 1p}
\end{equation*}
is finite, with appropriate modifications when $p=\infty$.
One of the standard classes of discrete weights is that of ``power'' type; given $s\in\mathbb R$, we define $w_s:\mathbb Z^d\to (0, \infty)$ by
\begin{align}
w_s(k)= \jb{k}^s
\label{power1}
\end{align}

\noi
 for  $k = (k_1,\dots, k_d)\in\mathbb Z^d$,
 where $\jb{k} = ( 1 + |k|^2)^\frac{1}{2}$
 with  $|k|^2=|k_1|^2+\cdots+|k_d|^2$.
For this particular discrete weight function $w_s$ in \eqref{power1},
we simply
 write $\l^p_s(\zd)$ for $\ell^p_{w_s}(\zd)$.
 When $p = 2$, we have $f \in \l^2_s$ if and only if $\F^{-1}(f) \in H^s(\T^d)$,
 where the latter space denotes
 the usual $L^2$-based Sobolev space\footnote{Recall that $H^s(\T^d) = L^2_s(\T^d)$,
 where $L^2_s(\T^d)$ is the Bessel potential space; see also \cite{BO5}.}
 and $\F^{-1}$ denotes the inverse Fourier transform.
 In  \cite{fg}, the space  $\l^2_s$ is referred to as ``discrete Sobolev space" (and denoted by $h^s$ in \cite{fg}).
 For general $1\le p \le \infty$, we have the following relation;
 $f \in \l^p_s$ if and only if $\F^{-1}(f)$ belongs
 to the so-called Fourier-Lebesgue space $\F L^{s. p}(\T^d)$ which plays
 an important role in the study of nonlinear PDEs; see, for example,
 \cite{OW1, OW2}.
 It is easy to see that  the dual of $\l^p_s$ is $\l^{p'}_{-s}$ where $p'$ is the H\"older conjugate of $p$.

Consider now an infinite tensor $\Si: \mathbb Z^d\times\mathbb Z^d\times\mathbb Z^d\to\mathbb C$; $\Si$ will be identified with the collection of its elements $\Si=\{\Si(j, k, \l)\}_{(j, k, \l)\in\mathbb Z^{3d}}$.
For appropriate sequence spaces $X, Y, Z$, a tensor
 $\Si$ induces an, a priori formally defined, bilinear operator $\cT_{\Si}: X\times Y\to Z$ acting on pairs of sequences $(f, g)\in X\times Y$ via the formula
\begin{equation}
\label{db1}
(\cT_{\Si} (f, g))_j=\sum_{k\in\mathbb Z^d}\sum_{\l\in\mathbb Z^d}\Si(j, k,\l)f_kg_\l,
\quad  j\in\zd.
\end{equation}
The operator in \eqref{db1} bears a strong resemblance to the continuous version of the bilinear 
Calder\'on-Zygmund singular integral operator in \eqref{kernel},
 which, under an appropriate condition, can be identified with
  the bilinear pseudodifferential operator
 in~\eqref{continuous-bilinearop}.
 Naturally, one is interested in conditions on the tensor $\Si$ that make the bilinear operator $\cT_\Si$ continuous on appropriate function spaces $X, Y, Z$. For example, assuming that
 $\|\Si\|_{\ell^1_j\ell^2_k\ell^2_\l}<\infty$, we have that $\cT_{\Si}: \ell^2(\zd)\times\ell^2(\zd)\to\ell^1({\zd})$ is a bounded bilinear operator. Indeed,
 by the Cauchy-Schwarz inequality, we have
\begin{align*}
\|\cT_{\Si}(f, g)\|_{\ell^1}&=\sum_{j\in\zd}|(\cT_{\Si}(f, g))_j|
\le \sum_{j \in \Z^d} \sum_{k. \l \in \Z^d} |\Si(j, k , \l)| |f_k| |g_\l|\\
&\leq \|\Si\|_{\ell^1_j\ell^2_k\ell^2_\l}\|f\|_{\ell^2}\|g\|_{\ell^2}.
\end{align*}

\noi
One can easily extend this argument  to more general H\"older triples of exponents $(p, q, r)\in [1, \infty]^3$ satisfying $\frac{1}{p}+\frac{1}{q}=\frac{1}{r}$; for example, assuming that the mixed Lebesgue norm\footnote{By using duality, we can take
the mixed-Lebesgue $\ell^r_j \ell_k^{p'}\ell^{q'}_{\l}$-norm
in any order and thus it suffices to assume that the minimum
mixed-Lebesgue norm is finite to guarantee the boundedness
of the operator $\cT_{\Si}: \ell^p(\zd)\times\ell^q(\zd)\to\ell^r(\zd)$.}
$\|\Si(j, k, \l)\|_{\ell^r_j \ell_k^{p'}\ell^{q'}_{\l}}<\infty$,  we can prove that
$\cT_{\Si}: \ell^p(\zd)\times\ell^q(\zd)\to\ell^r(\zd)$ is a bounded bilinear operator.

\begin{remark}\rm
Given an infinite matrix $\s:\Z^d \times \Z^d \to \C$, we can (formally) define
a discrete linear operator $\mathcal{L}_\s$ by
\begin{align}
(\mathcal L_\sigma(f))_j=\sum_{k\in\zd}\sigma(j, k)f_k, \quad j\in\zd,
\label{L1}
\end{align}

\noi
for $f = \{f_k\}_{k \in \Z^d}$.
Let $F$ be a function on the torus $\T^d = (\R/\Z)^d$
such that its Fourier coefficient $\ft F(k)$ agrees with $f_k$ for any $k \in \Z^d$.
Then, we can associate the operator $\mathcal{L}_\s$, acting on sequences,
with the following operator, acting on functions on $\T^d$:
\begin{align}
L_\s (F)(x) = \int_{\T^d} K_\s(x, y) F(y) dy,
\label{L2}
\end{align}

\noi
where the kernel $K_\s$ is given by
\[ K_\s(x, y) = \sum_{j \in \Z^d} \sum_{k \in \Z^d} \s(j, k)e^{2\pi i j \cdot x} e^{-2\pi i k\cdot  y}.\]

\noi
Hence,
prior knowledge on continuous linear operators on $\T^d$ of the form \eqref{L2}
provides insights on discrete linear operators on $\Z^d$ of the form \eqref{L1}.

Given
 $f = \{f_k\}_{k \in \Z^d}$ and  $g = \{g_k\}_{k \in \Z^d}$,
let $F$ and $G$ be functions on $\T^d$
such that $\ft F(k) = f_k$ and $\ft G(k) = g_k$
 for any $k \in \Z^d$.
Then, we can express a discrete bilinear operator $\cT_\Si$ in \eqref{db1}
as a continuous bilinear operator $T_\Si$ on $\T^d$, formally  given by
\begin{align*}
T_\Si(f, g)(x)=\int_{\T^d}\int_{\T^d}K_\Si(x, y, z)f(y)g(z)\,dydz,
\end{align*}

\noi
where the kernel $K_\Si$ is given by
\[ K_\Si(x, y, z) = \sum_{j \in \Z^d} \sum_{k \in \Z^d} \sum_{\l \in \Z^d}
\Si(j, k, \l)e^{2\pi i j \cdot x} e^{-2\pi i k\cdot  y}e^{-2\pi i \l\cdot  z}.\]

\noi
Thus, boundedness of $\cT_\Si$
from $\l^p_{s_1}(\Z^d) \times \l^q_{s_2} (\Z^d)$ to $\l^r_{s_3}(\Z^d)$
is equivalent to
boundedness of $T_\Si$
on the Fourier-Lebesgue spaces: $\F L^{s_1, p}(\T^d) \times
\F L^{s_2, q}(\T^d) \to \F L^{s_3, r}(\T^d)$.

\end{remark}

\section{Towards a class of discrete bilinear symbols}
\label{SEC:3}

The discussion at the end of the previous section makes it clear that a sufficiently fast decay in each one of the entries of the infinite tensor $\Si$ is sufficient to guarantee  boundedness of the associated bilinear operator $\cT_{\Si}$ on products of discrete Lebesgue spaces.
A natural question then is whether one can impose
an appropriate condition on the ``discrete kernel'' $\Si$ reminiscent of those for the continuous bilinear H\"ormander classes defined in \eqref{hormander}
(and  for their distributional kernels \eqref{diag-esti})
that would guarantee boundedness of the discrete bilinear operator $\cT_{\Si}$
defined in \eqref{db1}. 

One of the first observations towards answering this question is the following;
if  a tensor  $\Si$ is almost diagonal, 
namely, if $\Si$ decays rapidly away from the main diagonal
$\{(j, k, \l)\in\bz^{3d}: j=k=\l\}$,
then $\cT_\Si$ is bounded from
$ \ell^p(\zd)\times\ell^q(\zd)$ to $\ell^r(\zd)$. One of the goals of this section is to make this remark precise. It is worth noting that the sufficiency of almost diagonal conditions is rather natural
in view of  the estimates~\eqref{diag-esti} on distributional kernels corresponding to the
H\"ormander  classes $BS_{1, \delta}^0, 0\leq\delta<1,$ as well as their appearance elsewhere in multilinear harmonic analysis; see, for example, \cite{bt, gt} for almost diagonal estimates stemming from wavelet discretizations of multilinear operators.

Let $\om\in\mathbb R$ and $N\in\bn$. Given a tensor $\Si:\Z^d \times \Z^d \times \Z^d \to \C$, define the following norm:
\beq
\label{A1}
\|\Si\|_{\om, N}:=\sup_{j, k, \l\in\zd}\frac{|\Si(j, k, \l)|\jb{|j-k|+|j-\l|}^{2N}}{\jb{|j|+|k|}^\om\jb{|j|+|\l|}^\om}.
\eeq

\noi
Then, we have the following boundedness of the bilinear operator $\cT_\Si$.

\begin{proposition}
\label{PROP:1}
Let $s_1, s_2, \omega\in\br$ and set
\begin{align}
N_0=N_0(d, \om, s_1, s_2)
:=d+\om_++ \tfrac 12 \big\{|s_1+\om| + |s_2+\om|\big\},
\label{N0}
\end{align}
where $\om_+ := \max(\om, 0)$.  Suppose  that $\|\Si\|_{\om, N}<\infty$ for some $N >  N_0$.
Then,
 $\cT_\Si$ defined in~\eqref{db1} is a bounded bilinear operator from $\l^p_{s_1+\om}(\zd)\times \l^q_{s_2+\om}(\zd)$ to $\l^r_{s_1+s_2}(\zd)$
 for any $1\leq p, q, r\leq\infty$ with  $\frac{1}{p}+\frac{1}{q}=\frac{1}{r}$.

\end{proposition}

\begin{remark}\rm
Proposition \ref{PROP:1} is of particular interest when $s_1 = s_2 = \om =:\frac{s}{2} $.
In this case,
for $N > d +  \frac 12 s_+ + |s|$,
Proposition \ref{PROP:1} implies boundedness
of $\cT_\Si$ from $\l^p_s(\Z^d)\times \l^q_s(\Z^d)$ to $\l^r_s(\Z^d)$.
In particular, when $s_1 = s_2 = \om = 0$,
the hypothesis of Proposition \ref{PROP:1} reduces to
\beq
\label{A2}
\|\Si\|_{0, N}=\sup_{j, k, \l\in\zd}|\Si(j, k, \l)|\jb{|j-k|+|j-\l|}^{2N} < \infty
\eeq

\noi
for some $N > d$,
which resembles the bound \eqref{diag-esti} in the continuous case
(with $\alpha = \beta = \gamma = 0$).
See also Corollary \ref{COR:2} below.
\end{remark}

\begin{proof}[Proof of Proposition \ref{PROP:1}]
Given $f=\{f_k\}_{k\in\zd}\in \l^p_{s_1+\om}$ and
$g=\{g_\l\}_{\l\in\zd}\in \l^q_{s_2+\om}$,
set $F=\{F_k\}_{k\in\zd}$ and $G=\{G_\l\}_{\l\in\zd}$
by $F_k =\jb{k}^{s_1+\om}|f_k|$ and
$G_\l =\jb{\l}^{s_2+\om}|g_\l|$
such that   $F\in\ell^p$ and $G\in\ell^q$
with $\|F\|_{\elp}=\|f\|_{\l^p_{s_1+\om}}$
and $\|G\|_{\elq}=\|g\|_{\l^q_{s_2+\om}}$.

From  the definition \eqref{A1}, we have
\begin{align}
\|\cT_{\Si}( & f, g) \|_{\l^r_{s_1+s_2}}^r=\sum_{j\in\zd}\Big| \jb{j}^{s_1+s_2}
\sum_{k, \l\in\zd}\Si(j, k, \l)f_kg_\l\Big|^r \notag \\
&\le \sum_{j\in\zd } \Big(\sum_{k, \l\in\zd}|\Si(j, k, \l)|
\jb{j}^{s_1+s_2}
\jb{k}^{-s_1-\om}\jb{\l}^{-s_2-\om}F_k G_\l\Big)^r  \label{A3}\\
&\lesssim
\|\Si\|_{\om, N}^r
\sum_{j\in\zd}\bigg(\sum_{k, \l\in\zd}
\frac{\jb{|j|+|k|}^\om\jb{|j|+|\l|}^\om\jb{j}^{s_1+s_2}}
{\jb{|j-k|+|j-\l|}^{2N}\jb{k}^{s_1+\om}\jb{\l}^{s_2+\om}}F_kG_\l
\bigg)^r.
\notag
\end{align}

\noi
By  the triangle inequality, we have
$\jb{|j|+|k|}\lesssim \jb{j}\jb{j-k}$.
Combining this with a trivial bound
$\jb{|j|+|k|}^{-1}\lesssim \jb{j}^{-1}$, we obtain
\begin{align}
\jb{|j|+|k|}^\om&\lesssim \jb{j}^\om \jb{j-k}^{\om_+},
\label{A4}
\end{align}

\noi
where $\om_+ = \max(\om, 0)$.
Similarly, we have
\begin{align}
\jb{|j|+|\l|}^\om&\lesssim \jb{j}^\om \jb{j-\l}^{\om_+}.
\label{A4a}
\end{align}

\noi
By a version of Peetre's inequality \cite{Tr}, we have
\begin{align}
\jb{j}^{s_i+\om}&\lesssim
\min\big(\jb{k}^{s_i+\om} \jb{j-k}^{|s_i+\om|},
 \jb{\l}^{s_i+\om} \jb{j-\l}^{|s_i+\om|}\big), \quad i = 1, 2.
\label{A4b}
\end{align}

\noi
In view of the condition $N >N_0$, write $2N = N_1 + N_2$,
where
\begin{align}
N_i > d + \om_+ + |s_i + \om|, \quad i = 1, 2.
\label{A4c}
\end{align}

\noi
Then,  from \eqref{A3}, \eqref{A4}, \eqref{A4a},   and \eqref{A4b}
with a trivial bound
$\jb{j- k}^{N_1} \jb{j-\l}^{N_2} \le \jb{|j-k|+|j-\l|}^{2N}$,
we have
\begin{align}
\|\cT_{\Si}(f, g)\|_{\l^r_{s_1+s_2}}\lesssim
\|\Si\|_{\om, N}\bigg(\sum_{j\in\zd}(a_jb_j)^r\bigg)^{\frac 1r},
\label{A5}
\end{align}
where
\begin{align*}
a_j& =\sum_{k\in\zd} \jb{j-k}^{\om_+ + |s_1+\om| -N_1} F_k, \\
b_j &  =\sum_{\l \in\zd} \jb{j-\l}^{\om_+ +|s_2+\om| -N_2} G_\l.
\end{align*}

Let $a=\{a_j\}_{j\in\zd}$ and $b=\{b_j\}_{j\in\zd}$.
Then, it follows from \eqref{A5}, H\"older's inequality, Young's inequality (for a discrete convolution),
and \eqref{A4c}  that
\begin{align*}
\|\cT_{\Si}(f, g)\|&_{\l^r_{s_1+s_2}}
\lesssim \|\Si\|_{\om, N}\|a\|_{\elp}\|b\|_{\elq}\\
&=  \|\Si\|_{\om, N}\|\jb{\,\cdot\,}^{\om_++|s_1+\om|-N_1}*F\|_{\elp}
\|\jb{\,\cdot\,}^{\om_++|s_2+\om|-N_2}*G\|_{\elq}\\
&\lesssim \|\Si\|_{\om, N}\|f\|_{\l^p_{s_1+\om}}\|g\|_{\l^q_{s_2+\om}}.
\end{align*}

\noi
This proves Proposition \ref{PROP:1}.
\end{proof}

When $\om = 0$, we have the following corollary.

\begin{corollary}\label{COR:2}
Let $s_1, s_2\in\br$.
Suppose that $\Si$ satisfies \eqref{A2} for some $N > d
+ \frac12 \big\{|s_1|+ |s_2|\big\}$. Then, $\cT_\Si$ is a bounded bilinear operator from $\l^p_{s_1}(\zd)\times \l^q_{s_2}(\zd)$ to $\l^r_{s_1+s_2}(\zd)$
 for any $1\leq p, q, r\leq\infty$ with  $\frac{1}{p}+\frac{1}{q}=\frac{1}{r}$.
 In particular, if $\Si$ satisfies~\eqref{A2} for some $N > d$, then $\cT_{\Si}$
 is bounded from $\elp(\zd)\times\elq(\zd)$ to $\elr(\zd)$.
\end{corollary}

\begin{remark} \rm
(i)
We can slightly change the norm in \eqref{A1} on the infinite tensor $\Si$ and impose
a  more general condition $\|\Si\|_{\om_1, \om_2, N}<\infty$ for $\om_1, \om_2\in\br$ and some appropriately large $N$, where now
$$\|\Si\|_{\om_1, \om_2, N}:=\sup_{j, k, \l\in\zd}\frac{|\Si(j, k, \l)| \jb{|j-k|+|j-\l|}^{2N}}
{\jb{|j|+|k|}^{\om_1}\jb{|j|+|\l|}^{\om_2}}.$$

\noi
Essentially the same argument as in the  proof of  Proposition~\ref{PROP:1}
  yields  boundedness of the corresponding bilinear operator $\cT_{\Si}: \l^p_{s_1+\om_1}(\zd)\times \l^q_{s_2+\om_2}(\zd)\to \l^r_{s_1+s_2}(\zd).$

\smallskip

\noi
(ii)
From the perspective of \cite[Definition 2.1]{fg}, it is also natural to consider the following norm on the infinite tensor $\Si$:
\beq
\label{order00-seminorm}
||\Si||_{0, 0, \om, N}:=\sup_{j, k, \l\in\zd}\frac{|\Si(j, k, \l)|\jb{|j - k|+|j-\l|}^{2N}}{\jb{|j|+|k|+|\l|}^\om}.
\eeq
We will comment further on \eqref{order00-seminorm}
in Section \ref{SEC:5}.

\end{remark}

\begin{remark} \rm
We wish to end this section by discussing the appropriateness of the condition $\|\Si\|_{\om, N}<\infty$ in Proposition~\ref{PROP:1}. As mentioned in the introduction, the insight in \cite{C} had to do with finding an appropriate notion of order $\om$ for an infinite matrix $\sigma:\zd\times\zd\to\mathbb C$. This notion requires further defining the finite difference operator $\Delta^\alpha, \alpha\in\zd$;
see Section \ref{SEC:5}. However, for the purposes
of boundedness of the associated (linear) operator $\mathcal L_\sigma: \l^2_s(\zd)\to \l^2_{s-\om}(\zd)$,
defined in \eqref{L1},  we only need
 the condition
\beq
\label{linear-cond}
|\sigma (j, k)|\lesssim \jb{|j|+|k|}^\om \jb{j-k}^{-M}
\eeq
for all $j, k\in\zd$ and for some $M$ sufficiently large;
see \cite[Lemma 2.2]{fg}.
A natural way to bilinearize such a linear operator is to consider the tensor operator $\cT_{\Si}=\mathcal L_{\sigma_1}\otimes\mathcal L_{\sigma_2}$ with $\sigma_1, \sigma_2$ satisfying
the condition  \eqref{linear-cond}; that is, for all $j\in\zd$, we set
\begin{align*}
(\cT_\Si (f, g))_j&=(L_{\sigma_1}(f))_j(L_{\sigma_2}(g))_j\\
&=\sum_{k\in\zd}\sigma_1(j, k)f_k\sum_{\l\in\zd}\sigma_2(j, \l)g_\l\\
&=\sum_{k,\l\in\zd}\underbrace{\sigma_1(j, k)\sigma_2(j, \l)}_{=: \, \Si(j, k, \l)}f_kg_\l.
\end{align*}

\noi
It is easy to see that, if  $\sigma_1$ and $\sigma_2$ satisfy \eqref{linear-cond},
then $\Si$ in the last summation above satisfies $\|\Si\|_{\om, N}<\infty$
in \eqref{A1} with $N = \frac M2$.
\end{remark}

\section{Smoothing of discrete bilinear commutators}

Bilinear commutators are natural objects to consider within harmonic analysis.
Given a bilinear operator $T$ and an appropriate function $b$, one can consider the following {\it bilinear commutators}:
\begin{align}
\begin{split}
[T, b]_1(f, g)&= T(bf,g)-bT(f,g),\\
[T, b]_2 (f, g)&=T(f,bg)-bT(f,g).
\label{BB1}
\end{split}
\end{align}
Just as their linear counterparts, these commutators have a smoothing effect. More precisely, it was shown in \cite{BT13} that if $T$ is a bilinear Calder\'on-Zygmund operator and $b\in CMO(\R^d)$, then $[T, b]_i$, $i=1, 2$,  are compact bilinear operators from $L^p(\R^d)\times L^q(\R^d)\to L^r(\R^d)$ for all
$1<p,q<\infty$ and  $1\leq r <\infty$, satisfying  $\frac 1p + \frac 1q  = \frac 1r$.
See \cite{BO2} for a related discussion.
  The definition of a bilinear compact operator is most natural; given three normed spaces $X, Y,$ and $ Z$,
  we say that a bilinear operator $T: X\times Y\to Z$ is \emph{compact} if the set $\{T(f, g): \|f\|_X, \|g\|_Y\leq 1\}$ is pre-compact in $Z$.  Clearly, compactness of $T$ implies the continuity of $T$. The proof of the compactness statement mentioned above and other subsequent compactness results in the literature make use of the so-called Fr\'echet-Kolmogorov-Riesz theorem which provides a characterization of pre-compactness in the $L^p$-spaces; see, for example, Yosida's book \cite{Yo}. The version of this theorem for the $\l^p$ spaces \cite[Theorem 4]{HOH} reads as follows.

\begin{lemma}
\label{lp-compact}
A subset $\mathcal F\subset \l^p(\Z^d) , 1\leq p<\infty,$ is totally bounded if and only if the following two conditions are satisfied:


\vspace{0.5mm}

\begin{itemize}
\item[\textup{(i)}]	$\F$ is pointwise bounded,


\vspace{0.5mm}

\noi
\item [\textup{(ii)}]
Given any  $\eps>0$,  there exists $j_0 \in\mathbb N$ such that
\[\bigg(\sum_{|j|>j_0}|f_j|^p\bigg)^\frac 1p <\eps\]

\noi
for any $f=\{f_j\}_{j\in\zd}\in\mathcal F$.

\end{itemize}

\end{lemma}

We now define $\l^\infty_c(\Z^d)$ and $c_0(\Z^d)$ by
\begin{align*}
\l_c^\infty(\Z^d) &  =\big\{ \{b_k\}_{k \in \Z^d}:
b_k = 0 \text{ for  all but a finite number of $k$}\}, \\
c_0(\Z^d) & = \big\{ \{b_k\}_{k \in \Z^d}: b_k \to 0 \text{ as } |k| \to \infty\big\}.
\end{align*}

\noi
Recall from \cite[Example III.1.3]{RS1} that
$c_0 (\Z^d)$ is the completion of $\l^\infty_c(\Z^d)$
with respect to the $\l^\infty$-norm.
In the following, we work with $c_0(\Z^d)$.
In the context of discrete bilinear operators,
 this is a natural space to study commutators with a sequence~$b$ if we recall that in the continuous case, the space  $CMO(\R^d)$ is the closure of $C^\infty_c(\rd)$ with respect to  the $BMO$ (bounded mean oscillations)-norm.
 Our main result is the following.

\begin{theorem}
\label{THM:3}
Suppose that  $\Si$ satisfies~\eqref{A2} for some $N > d$ and $b\in c_0(\zd)$.
Let $\cT_\Si$ be as in \eqref{db1}.
Then, the bilinear commutators
$[\cT_{\Si}, b]_i$, $i=1, 2$, defined as in \eqref{BB1},   are compact bilinear operators from $\elp(\zd)\times\elq(\zd)$ to $\elr(\zd)$
 for any $1\leq p, q\leq\infty$ and $1\leq r<\infty$ with  $\frac{1}{p}+\frac{1}{q}=\frac{1}{r}$.
\end{theorem}

\begin{proof}
We will only consider the first commutator $[\cT_{\Si}, b]_1$.
The calculations for $[\cT_{\Si}, b]_2$ are similar.
Given any $b\in\l^\infty(\Z^d)$, $f\in\l^p(\Z^d)$, and $g\in\l^q(\Z^d)$,
it follows from  Corollary~\ref{COR:2} (with $N > d$)
and H\"older's inequality that
\begin{align}
\|[\cT_{\Si}, b]_1(f, g)\|_{\l^r}\lesssim \|b\|_{\l^\infty}\|f\|_{\l^p}\|g\|_{\l^q}.
\label{BB2}
\end{align}

\noi
That is, $[\cT_{\Si}, b]_1: \l^p(\zd)\times \l^q(\zd)\to\l^r(\zd)$ is bounded with  the operator norm $\|[\cT_{\Si}, b]_1\|\lesssim \|b\|_{\l^\infty}$. In particular, this proves condition (i) in Lemma~\ref{lp-compact}.
Hence, it remains to prove the condition (ii).
Without loss of generality, assume $b \ne 0$.
Let $f = \{f_k\}_{k \in \Z^d}$ and $g = \{g_k\}_{k \in \Z^d}$
such that $\|f\|_{\l^p},  \|g\|_{\l^q}\leq 1$.
Then, our goal is to show that, given $\eps > 0$,
there exists $j_0=  j_0(\eps) \in \N$ (independent of $f$ and $g$) such that
\begin{equation}
\label{BB3}
\bigg(\sum_{|j|>j_0}\big|([\cT_{\Si}, b]_1(f, g))_j\big|^r\bigg)^\frac 1r \lesssim\eps.
\end{equation}

\noi
In view of \eqref{BB2}
and the density of $\l^\infty_c(\Z^d)$ in $c_0(\Z^d)$
with respect to the $\l^\infty$-norm,
it suffices to prove \eqref{BB3}
for $b \in \l^\infty_c(\Z^d)$.
Thus, we assume that there exists $j_1\in\mathbb N$ such that
\begin{align}
b_j= 0
\label{BB4}
\end{align}

\noi
for any $|j| > j_1$.

Formally, we have
\[
([\cT_{\Si}, b]_1(f, g))_j=\sum_{k\in\zd}\sum_{\l\in\zd}\Si(j, k, \l)(b_k-b_j)f_kg_{\l}.
\]
It follows from \eqref{BB4}, \eqref{A2}, and Young's and H\"older's inequalities that, for $|j| > j_ 0 \gg j_1$, we have
\begin{align*}
|([\cT_{\Si}, b]_1(x, y))_j|&\leq \sum_{k\in\zd}\sum_{\l\in\zd}
\ind_{|k|\le j_1} \cdot
|\Si(j, k, \l)||b_k||f_k||g_{\l}|\\
&\lesssim \|b\|_{\l^\infty}\sum_{|k|\le j_1}\frac{|f_k|}{\jb{j-k}^N}
 \sum_{\l\in\zd}
\frac{|g_\l|}{\jb{j-\l}^N}\\
&\lesssim \|b\|_{\l^\infty}\jb{j}^{-N} \Big(\sum_{|k|\leq j_1}|f_k|\Big)
\|g\|_{\l^\infty}\\
&\lesssim \|b\|_{\l^\infty}\jb{j}^{-N} j_1^{\frac d{p'}} \|f\|_{\l^p}
\|g\|_{\l^q}.
\end{align*}

\noi
Hence, we have
\begin{align*}
\big\| \ind_{|j|  >j_0}\cdot ([\cT_{\Si}, b]_1(f, g))_j\big\|_{\l^r}
& \les \|b\|_{\l^\infty}
j_1^{\frac d{p'}}
\| \ind_{|j|  >j_0} \cdot \jb{j}^{-N} \|_{\l^r}\\
& \les \|b\|_{\l^\infty}
j_1^{\frac d{p'}} \jb{j_0}^{-N + \frac d{r}}
< \eps
\end{align*}

\noi
by choosing $j_0 \gg j_1$; in the calculations above, $\ind_X$ denotes the characteristic function of the set $X$. This proves \eqref{BB3} and therefore completes the proof of Theorem~\ref{THM:3}.
\end{proof}

\section{Tensors of order $\omega$}
\label{SEC:5}


In Section~\ref{SEC:3}, we studied  the boundedness property
of the discrete bilinear operator $\cT_\Si$ for an infinite tensor $\Si:\zd\times\zd\times\zd\to\C$ with a finite $\|\cdot\|_{\om, N}$-norm defined in
\eqref{A1}.
In this section, we seek for an analogous definition for infinite tensors as the one given in \cite{C} for infinite matrices, and possible examples of such tensors. Following \cite{C}, we first define   partial finite difference operators on the set of infinite
tensors  $\{\Si(j, k, l)\}_{(j, k, l)\in\mathbb Z^{3d}}$; see also \cite[Definition 2.1]{fg}. For $m\in\{1,\dots, d\}$, let $e_m=\{\delta_{mn}\}_{n\in\zd},$ where $\delta_{mn}$ is the Kronecker symbol.
Given a tensor $\Si$, we denote by $\Si_2^{m, +}$ and $\Si_2^{m, -}$ the shifted  tensors defined by
$$\Si_2^{m, +}(j, k, \l)=\Si(j+e_m, k+e_m, \l)
\quad \text{and}\quad \Si_2^{m, -}(j, k, \l)=\Si(j-e_m, k-e_m, \l).$$

\noi
Then, we  define the partial finite difference operators $\Dl_2^{m, +}$ and $\Dl_2^{m, -}$ acting on $\Si$:
$$\Delta_2^{m, +}\Si:=\Si_2^{m, +}-\Si
\quad \text{and}\quad \Delta_2^{m, -}\Si:=\Si_2^{m, -}-\Si.$$

\noi
Similarly, we define $\Si_3^{m, \pm}$ and $\Delta_3^{m, \pm}\Si$ by
$$\Si_3^{m, \pm}(j, k,\l)=\Si(j\pm e_m, k, \l\pm e_m)
\quad \text{and} \quad \Delta_3^{m, \pm}\Si:=\Si_3^{m, \pm}-\Si.$$

Let $i = 2, 3$.
Then, for $t\in\bz$,  we set  $\Delta_{i, m}^t=(\Delta_i^{m, \text{sign}(t)})^{|t|}$ and $\Delta_{i, m}^0=\text{Id}$. Finally, for $\alpha=\{\alpha_m\}_{m=1}^d\in\zd$,
we set
$$\Delta_{i}^\alpha=\Delta_{i, 1}^{\alpha_1}\cdots\Delta_{i, d}^{\alpha_d}.$$

\begin{definition}
\label{discrete-CM} \rm
Let $\om\in\br$ and $N \in \N$. We say that $\Si$ belongs to the class $BT^{\om, N}(\zd)$ if, for all $\alpha, \beta\in\zd$,  there exists $C_{N, \al, \be} > 0$ such that
\beq
\label{dis-horm}
|\Delta_2^\alpha\Delta_3^\beta \Si(j, k, \l)|\le C_{N, \alpha, \beta}
\jb{|j|+|k|+|\l|}^{\om-|\alpha|-|\beta|}\jb{|j-k| + |j-\l|}^{-2N}
\eeq
\noi
for any  $j, k, \l\in\zd$.

\end{definition}


Clearly, given $N \in\N$, the bilinear tensor classes $BT^{\om, N}(\Z^d)$ of order $\omega$ are nested, that is, if $\omega_1\leq\omega_2$, then $BT^{\omega_1, N}(\zd)\subseteq BT^{\omega_2, N}(\zd)$.
Note that $BT^{\om, N}(\zd)$ is a Fr\'echet space under the family of semi-norms:
\beq
\label{orderab-seminorm}
\|\Si\|_{\alpha, \beta, \om, N}:=\sup_{j, k, \l\in\zd}
\frac{|\Dl_2^\al \Dl_3^\be\Si(j, k, \l)| \jb{|j-k|+|j-\l|}^{2N}}{\jb{|j|+|k|+|\l|}^{\om-|\alpha|-|\beta|}}
\eeq

\noi
for  $\alpha, \beta\in\zd$.
When  $\alpha =  \beta = 0$ (as elements in $\zd$), \eqref{orderab-seminorm}
reduces to the expression in~\eqref{order00-seminorm}.
Let us point out that, even when $\al = \be = 0$,
 \eqref{order00-seminorm} and  \eqref{A1} are, in general,  not comparable. However, we have the following estimates:

 \smallskip
\begin{itemize}
\item If $\om\geq 0$, then $\|\Si\|_{0, 0, 2\om, N}\lesssim \|\Si\|_{\om, N}\lesssim\|\Si\|_{0,0, \om, N},$

 \smallskip

\item If $\om\leq 0$, then $\|\Si\|_{0, 0, \om, N}\lesssim \|\Si\|_{\om, N}\lesssim\|\Si\|_{0,0, 2\om, N}.$
\end{itemize}

 \smallskip

\noi
An immediate consequence of Proposition~\ref{PROP:1} is given below.

\begin{corollary}
\label{COR:3}
Given  $s_1, s_2, \omega\in\br$ and $N \in \N$,
suppose that

\smallskip

\begin{itemize}
\item if $\om \ge 0$, then
$\Si\in BT^{\om, N}(\zd)$ for some $N > N_0$, where $N_0$
is as in \eqref{N0},

\smallskip

\item if $\om <  0$, then
$\Si\in BT^{2\om, N}(\zd)$ for some $N > N_0$.

\end{itemize}

\noi
Let $\cT_\Si$ be as in \eqref{db1}. Then,
$\cT_\Si$ is a bounded bilinear operator from $\l^p_{s_1+\om}(\zd)\times \l^q_{s_2+\om}(\zd)$ to $\l^r_{s_1+s_2}(\zd)$
for any  $1\leq p, q, r\leq\infty$ with  $\frac{1}{p}+\frac{1}{q}=\frac{1}{r}$.
In particular, if
we define the bilinear tensor class $BT^0(\zd)$ of order zero by
\[BT^0(\zd) := \bigcup_{N > d} BT^{0, N}(\zd), \]
then $\cT_\Si$ is a bounded bilinear operator from $\l^p(\zd)\times \l^q(\zd)$ to $\l^r(\zd)$
for any $\Si\in BT^0(\zd)$.
\end{corollary}

Let us provide simple examples of bilinear tensors of order zero.
Consider the following tensor:
\[ \Si_1(j, k, \l)=\theta_j\cdot\ind_{j = k=\l}\cdot \ind_{|j|+|k|+|\l|\le M}\]

\noi
with $\{\theta_j\}_{j\in\zd}\in\l^{\infty}(\zd)$
and some $M \in\mathbb N$.  Then, it is easy to see that $\Si_1\in BT^0(\zd)$.
Next, given a function $\Phi: \Z^d \times \Z^d  \to \C$	
with a bound $|\Phi(x, y)| \les \jb{|x| + |y|}^{-N}$
for any $x, y \in \Z^d$ and
for some $N > d$,
consider
\[ \Si_2(j, k, \l)=\Phi(j - k, j - \l) \cdot\ind_{j = k+\l}.\]

\noi
Noting that $\Dl^\al_2 \Dl^\be_3 \Si_2 (j, k, \l) = 0$ unless $\al = \be = 0$,
we see that $\Si_2\in BT^0(\zd)$.

\medskip

We also state a compactness result
on the bilinear commutators of a discrete bilinear operator $\cT_\Si$
with
 $\Si\in BT^0(\zd)$ and $b\in c_0(\zd)$,
 which follows from Theorem~\ref{THM:3}
and Corollary~\ref{COR:3}.

\begin{corollary}
\label{COR:4}
Suppose that $\Si\in BT^0(\zd)$ and $b\in c_0(\zd)$.
Let $\cT_\Si$ be as in \eqref{db1}. Then, the bilinear commutators
$[\cT_{\Si}, b]_i$, $i=1, 2$, defined as in \eqref{BB1}, are compact bilinear operators from $\elp(\zd)\times\elq(\zd)$ to $\elr(\zd)$
for any $1\leq p, q\leq\infty$ and $1\leq r<\infty$ with  $\frac{1}{p}+\frac{1}{q}=\frac{1}{r}$.
\end{corollary}

\begin{remark}[{\bf transposes}]\rm
The (formal) transposes of a (continuous) bilinear operator
$T:\mathcal S(\R^d)\times\mathcal S(\R^d)\to\mathcal S'(\R^d)$ are defined by the duality relations
\[
\langle T(f, g), h\rangle=\langle T^{*1}(h, g), f\rangle=\langle T^{*2}(f, h), g\rangle.
\]

\noi
If
 $T$ has kernel $K$ as in \eqref{kernel}, then its formal transposes
 $T^{*1}$ and $T^{*2}$ have kernels given by $K^{*1}(x, y, z)=K(y, x, z)$ and $K^{*2}(x, y, z)=K(z, y, x)$,
 respectively. In the discrete setting, for  $\cT_\Si:\elp(\Z^d)\times\elq(\Z^d)\to\elr(\Z^d)$
 with a discrete symbol $\Si$, using the natural duality pairing
\[\langle\cT (f, g), h\rangle=\sum_{j\in\zd}(\cT(f, g))_jh_j
= \sum_{j, k, \l\in\zd} \Si(j, k, \l) f_k \, g_\l \, h_j,\]

\noi
it is easy to see that,
given  $\Si\in BT^{\om, N} (\zd)$ for some $N \in \N$,
we have $(\cT_{\Si})^{*i}=\cT_{\Si^{*i}}$
with $\Si^{*i}\in BT^{\om, N} (\zd)$, 
$i=1, 2$,
where $\Si^{*1}(j, k, \l)  = \Si(k,j,  \l)$
and $\Si^{*2}(j, k, \l)  = \Si(\l, k, j)$.
This is not surprising considering that if $K$ is a bilinear Calder\'on-Zygmund kernel, then so are $K^{*i}, i=1, 2$; see also \cite[Theorem 2.1]{bmnt} for the symbolic calculus of the H\"ormader class $BS_{1, 0}^{\om}(\rd)$.
\end{remark}

We conclude this paper by presenting some other natural examples of infinite tensors belonging to the classes introduced in Definition~\ref{discrete-CM}.

\medskip

Given multi-indices $a, b \in(\bn\cup\{0\})^d$,
consider the following bilinear partial differential operator  on $\T^d$: 
\begin{align}
T_{a, b}(F, G) (x) =  \frac{\partial^a F}{\partial x^a}\frac{\partial^b G}{\partial x^b}, \quad x \in \T^d.
\label{D2}
\end{align}

\noi
By taking the Fourier transform,
we have
\begin{align*}
\F (T_{a, b}(F, G))(j) = \sum_{k \in \Z^d}\sum_{\l \in \Z^d} (2\pi i k)^a (2\pi i \l)^b \cdot \ind_{j = k+\l} \cdot f_k  \, g_\l ,
\end{align*}

\noi
where $f_k = \ft F(k)$ and $g_\l = \ft G(\l)$.
By setting
\begin{align}
\Si_{a, b}(j, k, \l) = (2\pi i k)^a(2\pi i \l)^b \cdot \ind_{j =  k+\l},
\label{D2a}
\end{align}
  we then have
$\cT_{\Si_{a, b}}(f, g) (j) = \F(T_{a, b}(F, G))(j)$, $j \in \Z^d$.
More generally, we can consider a tensor $\Si_\Phi$ of the form
\begin{align}
\Si_\Phi(j, k, \l) = \Phi(k, \l)\cdot  \ind_{j = k +  \l}.
\label{D3}
\end{align}

\begin{lemma}\label{LEM:X}
Let $\Phi \in C^\infty(\R^{d}\times \R^d; \C)$.
Suppose that  there exists $\om \in \R$
such that
\begin{align}
|\dd_x^\al \dd_y^\be \Phi(x, y)| \les \jb{|x| + |y|}^{\om - |\al| - |\be|}
\label{E0}
\end{align}

\noi
for any  multi-indices $\al, \be \in(\bn\cup\{0\})^d$ and $x, y \in \R^d$. Then, the tensor $\Si_\Phi$ defined in~\eqref{D3} belongs to $BT^{\om + 2N, N}(\Z^d)$ for any $N \in \N$. In particular, if $\Si_{a, b}$ is the tensor of the bilinear partial differential operator given  in~\eqref{D2a},
then $\Si_{a, b}\in BT^{|a| + |b| + 2N, N}(\Z^d)$.
\end{lemma}

\begin{proof} 

Under $j = k+  \l$, we have
\begin{align*}
\ind_{j = k +  \l} \cdot \frac{\jb{|j|+ |k| + |\l|}^{\om + 2N - |\al| - |\be |}}{\jb{|j - k| + |j - \l|}^{2N}}
& \sim \ind_{j = k +  \l} \cdot \jb{|j| + |k| + |j-k|}^{\om - |\al| - |\be|} \\
& \sim \ind_{j = k +  \l} \cdot \jb{ |k| + |j-k|}^{\om - |\al| - |\be|}
\end{align*}

\noi
and thus it suffices to show that,
for any $\al, \be \in \Z^d$, there exists $C_{ \al. \be } > 0$ such that
\begin{align}
|\Delta_2^\alpha\Delta_3^\beta \Si_\Phi(j, k, \l)|
\le C_{ \al, \be} \cdot \ind_{j = k +  \l} \cdot \jb{|k| + |j-k|}^{\om - |\al| - |\be|}
\label{E2}
\end{align}

\noi
for any $j, k, \l \in \Z^d$.
For simplicity of notation, we drop $ \ind_{j = k +  \l}$
but it is understood that $j = k + \l$  in the following.
Moreover, since the constant $C_{N, \al, \be}$ in \eqref{dis-horm} can depend on $\al$ and $\be$,
we only need to prove the bound \eqref{dis-horm} for
\begin{align}
|j|, |k|, |\l| \gg |\al| + |\be|.
\label{E3}
\end{align}

When $\al = \be = 0$, we have
\[  |\Phi(k, \l)| \les \jb{|k| + |\l|}^{\om},  \]

\noi
yielding \eqref{E2}.
Let $m = 1, \dots, d$. By the mean value theorem, we have
\begin{align*}
\Dl_{3}^{m, +} \Si _\Phi(j, k, \l)= \Phi(j - \l, \l + e_m) - \Phi(j - \l, \l)
= \dd_{y_m} \Phi(j- \l,  \l+ \eps e_m)
\end{align*}

\noi
for some $\eps \in [0, 1]$.
By iteratively applying difference operators with the mean value theorem, we have
\begin{align*}
\Dl_{3}^{\be} \Si_\Phi (j, k, \l)
& = \dd^\be_y  \Phi(j - \l, \l+ \eps_\be)\\
& =  \dd^\be_y  \Phi(k, j- k + \eps_\be)
\end{align*}

\noi
for some $|\eps_\be| \le |\be|$.
Now, by applying $\Dl_2^\al = \Dl_{2, 1}^{\al_1} \cdots \Dl_{2. d}^{\al_d}$
in an iterative manner together with the mean value theorem, we have
\begin{align*}
\Dl_{2}^{\al}\Dl_3^\be \Si_\Phi (j, k, \l)
& = \dd^\al_x  \dd_y^\be \Phi(k+ \eps_\al, j - k + \eps_\be)
\end{align*}

\noi
for some $|\eps_\al| \le |\al|$.
In view of \eqref{E0} and \eqref{E3}, we then obtain
\begin{align*}
|\Dl_{2}^{\al}\Dl_3^\be \Si_\Phi (j, k, \l)|
& \les \jb{|k+ \eps_\al| + |j - k+ \eps_\be|}^{\om - |\al| - |\be|}\\
& \les \jb{|k|+|j - k|}^{\om- |\al| - |\be|},
\end{align*}

\noi
which yields \eqref{E2}.
\end{proof}


\begin{example}\rm
Next, let us consider the multiplication operator $M_V$ on the physical side given by
$$M_V(F, G)(x) = V(x) F(x) G(x).$$
By taking the Fourier transform, we have
\[ \F(M_V(F, G)) =  \sum_{k, \l \in \Z^d} \ft V(j - k - \l) f_k g_\l.\]

%

\noi
Given $V \in C^\infty(\T^d)$, consider the tensor
$$\Si_V (j, k, \l) := \ft V(j - k - \l).$$
Note that we have
$\Delta_2^\alpha\Delta_3^\beta\Si_V(j, k, \l) = 0$
unless $\al = \be = 0$.
By the smoothness of $V$,
we have
\begin{align*}
|\Si_V(j, k, \l)| = |\ft V(j - k - \l)| \les \jb{j - k - \l}^{-K}
\end{align*}

\noi
for any $K > 0$.
However, note that, for $j = 2k = 2\l$, we have
 $ |\ft V(j - k - \l) | \les 1$, while
\begin{align*}
\jb{|j - k| + |j-\l|}^{-2N} \sim \jb{j}^{-2N} \longrightarrow 0
\end{align*}

\noi
as $|j| \to \infty$.
This shows that we have $\Si_V \in BT^{\om, N}(\Z^d)$
only for $\om \geq 2N$. Compare this with the linear case
studied in Lemma 2.7 (ii) in \cite{fg}.
\end{example}

\begin{example}\rm
Lastly, consider the tensor
$$\Si_{V, \Phi}(j, k, \l) = \ft V(j - k - \l) \Phi(k, \l),$$
where $V \in C^\infty(\T^d)$ and $\Phi \in C^\infty(\R^d \times \R^d; \C)$
satisfies \eqref{E0}.
The tensor for the
bilinear partial differential operator with variable coefficients
in \eqref{D1}
is given by a linear combination of such tensors.
By arguing as in the proof of Lemma \ref{LEM:X}, we have
\begin{align*}
\Dl_{2}^{\al}\Dl_3^\be \Si_{V, \Phi} (j, k, \l)
& = \ft V(j - k - \l) \dd^\al_x  \dd_y^\be \Phi(k+ \eps_\al, j - k + \eps_\be)
\end{align*}

\noi
for some $|\eps_\al| \le |\al|$ and  $|\eps_\be| \le |\be|$.
Without loss of generality, assume \eqref{E3}.
Then, we have
\begin{align*}
|\Dl_{2}^{\al}\Dl_3^\be \Si_{V, \Phi} (j, k, \l)|
 \les \jb{j - k - \l}^{-K}  \jb{|k|+|\l|}^{\om- |\al| - |\be|}.
\end{align*}

\noi
for any $K>0$.
When $|j| \les |k| + |\l|$,  we  have
\begin{align*}
|\Dl_{2}^{\al}\Dl_3^\be \Si_{V, \Phi} (j, k, \l)|
& \les  \jb{|j| + |k|+|\l|}^{\om- |\al| - |\be|}.
\end{align*}

\noi
On the other hand,
when $|j| \gg  |k| + |\l|$,  we  have
\begin{align*}
|\Dl_{2}^{\al}\Dl_3^\be \Si_{V, \Phi} (j, k, \l)|
& \les  \jb{j}^{-K+\om} \sim \jb{|j| + |k|+|\l|}^{-K+\om}
\end{align*}

\noi
for any $K > 0$.
Hence, we conclude that $\Si_{V, \Phi} \in BT^{\om + 2N, N} (\Z^d)$
for any $N \in \N$.
\end{example}

\begin{ackno}\rm
T.O.~was supported by the
 European Research Council (grant no.~864138 ``SingStochDispDyn").
T.O.~would like to thank 
the  Centre de recherches math\'ematiques,  Canada, 
for its hospitality, 
where this manuscript was prepared.

\end{ackno}

\noi
{\bf Conflict of interest:}
The authors have  no relevant financial or non-financial interests to disclose.
%
%

\end{document}